\documentclass{amsart}

\usepackage{pdfsync}
\usepackage{mathpazo}
\usepackage{amscd}
\usepackage{amsmath}
\usepackage{amssymb}
\usepackage{amsthm}
\usepackage{epsf}
\usepackage{latexsym}
\usepackage{verbatim}
\usepackage{bbm}
\usepackage[utf8]{inputenc}
\usepackage{pdfsync}
\usepackage{enumerate}
\usepackage{amsmath,amsthm,mathrsfs}
\usepackage[english]{babel}
\usepackage{graphicx}
\usepackage{amssymb, xcolor}
\usepackage{epstopdf, enumerate}
\usepackage{amsfonts,amssymb,amsthm}
\usepackage{pifont, url, lscape}
\usepackage{amsfonts,amstext,amssymb,verbatim,epsfig}
\usepackage{pgf,tikz}
\usepackage{float}
\usetikzlibrary{arrows}
\usepackage{hyperref}
\usepackage{cite}
\usepackage{epstopdf}
\usepackage{color}
\usepackage{transparent}
\usepackage{subfigure}
\usepackage{xcolor}
\usepackage[normalem]{ulem}

\usepackage[left=1.1in,top=1.5in,right=1in,bottom=1in]{geometry}

\newcommand{\Sym}{\text{Sym}}

\newcommand{\E}{\mathbb{E}}
\newcommand{\mS}{\mathbb{S}}

\newcommand{\R}{\mathbb{R}}

\newcommand{\C}{\mathbb{C}}

\newcommand{\Pa}{\mathbb{P}}

\newcommand{\Ind}{\textnormal{Ind}}

\numberwithin{equation}{section}

\hypersetup{colorlinks=true,pdfborder=000,  citecolor  = blue, citebordercolor= {magenta}}
\hypersetup{linkcolor=purple}
\makeatletter
\let\reftagform@=\tagform@
\def\tagform@#1{\maketag@@@{(\ignorespaces\textcolor{purple}{#1}\unskip\@@italiccorr)}}
\renewcommand{\eqref}[1]{\textup{\reftagform@{\ref{#1}}}}
\makeatother
\newcommand{\rf}[1]{(\ref{#1})}

\theoremstyle{plain}
\newtheorem{theorem}{Theorem}[section]
\newtheorem{corollary}[theorem]{Corollary}
\newtheorem{lemma}[theorem]{Lemma}
\newtheorem{proposition}[theorem]{Proposition}

\newtheorem{remark}[theorem]{Remark}

\newenvironment{Proof of lemma}{\noindent{\bf Proof of Lemma}}{\hfill$\Box$\newline}
\newenvironment{Proof of theorem}{\noindent{\bf Proof of Theorem}}{\hfill{\footnotesize${\square}$}\newline}
\newenvironment{Proof of theorems}{\noindent{\bf Proof of Theorems}}{\hfill$\Box$\newline}
\newenvironment{Proof of proposition}{\noindent{\bf Proof of Proposition}}{\hfill$\Box$\newline}
\newenvironment{Proof of propositions}{\noindent{\bf Proof of Propositions}}{\hfill$\Box$\newline}
\newenvironment{Proof of exercise}{\noindent{\it Proof of Exercise:}}{\hfill$\Box$}
\newenvironment{Acknowledgements}{\noindent{\bf Acknowledgements}}

\begin{document}
	
\title{Expected number of critical points of random holomorphic sections over complex projective space}\author{Xavier Garcia}
\address{Department of Mathematics\\
Northwestern University\\
Evanston, IL 60208 USA}
\email{sphinx@math.northwestern.edu}
\maketitle
	
\begin{abstract}
We study the high dimensional asymptotics of the expected number of critical points of a given Morse index of Gaussian random holomorphic sections over complex projective space. We explicitly compute the exponential growth rate of the expected number of critical points of the largest index and of diverging indices at various rates as well as the exponential growth rate for the expected number of critical points (regardless of index). We also compute the distribution of the critical values for the expected number of critical points of smallest index.
\end{abstract}
	
\section{Introduction and main results}
	
The purpose of this paper is to determine the high dimensional asymptotics of the expected number of critical points of a given Morse index of Gaussian $SU(m+1)$ multivariate polynomials of a fixed degree $N$ as the dimension $m$ tends to infinity. By definition, these critical points are those of the holomorphic sections of the line bundle $\mathcal{O}(N) \rightarrow \C \Pa^m$ equipped with the Chern connection induced from the Fubini-Study metric. The statistics of critical points of Gaussian random holomorphic sections have been studied extensively in Douglas, Shiffman, and Zelditch~\cite{DSZ1} \cite{DSZ}, mainly as a tool to understand the vacuum selection problem in string theory. The main focus of \cite{DSZ} (as well as most of the literature on Gaussian random holomorphic sections) is on the large degree limit,  namely as $N$ tends to infinity. In this paper, we adopt a different point of view and focus instead on high dimensional limits. In Baugher~\cite{BB}, it was proven that the number of critical points with Morse index close to $m$ (i.e. saddle points) grows exponentially. We not only recover this result, but also obtain estimates for the expected number of critical points, regardless of their indices. We will also compute the exact distribution of the critical values of index $m$, which recovers the formula in {\sc Theorem} 1.4 of Baugher~\cite{BB}.  
	
We will approach our problems by random matrix theory, in particular we will use large deviation results concerning the eigenvalues of a Wishart matrix ensemble, also known in the statistics literature as sample covariance matrices. This approach expands the connection between critical points of Gaussian fields and random matrix theory initiated in the seminal paper of Auffinger, Ben Arous and \v{C}ern\'y~\cite{ABC}, in which they established a link between critical points of isotropic Gaussian fields on the sphere and eigenvalues of the Gaussian orthogonal ensemble (GOE). The underlying reason for the success of random matrix theory in both areas is the existence of large symmetry groups, namely $SO(m+1)$ on $\mS^m$ and $SU(m+1)$ on $\C \Pa^m$. The techniques in  Auffinger and Ben Arous~\cite{AB} help address the related problem, namely the behavior of Gaussian random critical points of spherical harmonics on $\mS^m$ of large degrees, since the covariance kernel which arises there is also invariant under $SO(m+1)$ as in the spin glass case. We will say more on this later.
	
We now describe the setting and our main results. We consider the line bundle $\mathcal{O}(N)$ over $\C \Pa^m$ equipped with Fubini-Study metric $h$ and induced Chern connection $\nabla$. We endow the space of holomorphic sections $H^0(\C \Pa^m, \mathcal{O}(N))$ with the inner product induced by the metric, namely for two sections $s_1,s_2$ we set
$$ \langle s_1,s_2 \rangle := \int_{\C \Pa^m} h_z(s_1,s_2)v(dz),$$
where $v$ is the Fubini-Study volume element. We view $H^0(\C \Pa^m, \mathcal{O}(N))$ as a finite dimensional Hilbert space and choose an orthonormal basis $s^N_i$. With this basis, we can form the Gaussian field
\begin{equation}\label{eq:Gauss}
s = \sum_{i} c_i s^N_i,
\end{equation} 
where the $c_i$ are circularly symmetric complex Gaussians with the normalized variance
\begin{equation*}
\E|c_i|^2 = \frac{\text{Vol}(\C \Pa^m)}{\dim H^0(\C \Pa^m, \mathcal{O}(N))} = \frac{N! \pi^m}{(N+m)!}.
\end{equation*}
It is clear that the distribution of $s$ is independent of the choice of the orthonormal basis.
For any Borel set $B \subset \R_{+} = [0,\infty)$ and integer $m \leq k \leq 2m$, we consider $\mathcal{N}_{m,k,N}(B)(s)$, the number of critical points with Morse index $k$ for a section $s$ with $h_z(s,s) =||s(z)||_h^2 \in (m+1)B$; symbolically,
\begin{equation*}
	\mathcal{N}_{m,k,N}(B) = \sum_{z : \nabla s(z) = 0,\ \text{Ind}(\nabla^2s)(z) = k} \mathbbm{1}_{ (m+1)B} (||s(z)||_h^2 ),
\end{equation*}
where we understand $\text{Ind} (\nabla^2 s)$ as the index of the real Hessian of $\log ||s(z)||_h^2$.  Thus $\mathcal{N}_{m,k,N}(B)$ is an integer-valued random number if $s$ is sampled from the Gaussian field \rf{eq:Gauss}.  The random variable
$$\mathcal{N}_{m,N}(B) = \sum_k\mathcal{N}_{m,k,N}(B)$$ 
is the total number of critical points regardless of their Morse indices.
	
We will prove two types of asymptotics for $\mathcal{N}_{m,2m-k,N}(B)$ as the dimension $m$ goes to infinity: with fixed $k$ and with linearly growing $k$. More specifically, in the latter case we will consider the relation  of the form $k(m)/m \rightarrow  \gamma \in (0,1)$ as $m\rightarrow\infty$.

We now state our main results. For a given $\gamma\in (0,1)$ define $s_{\gamma}$ by 
\begin{equation}
\int_{s_{\gamma}}^{4} f_{MP}(x) dx = \gamma, \label{cond}
\end{equation}
where
$$f_{MP}(x) =\frac{ \sqrt{(4-x)x}}{2 \pi x}$$
is the Marchenko-Pastur density function on $[0, 4]$. 

Our first main result concerns the exponential growth rate of the expected number $\E\, \mathcal{N}_{m,2m-k,N}(x,\infty)$ of critical points.
	
\begin{theorem}\label{mainthm}  Fix an integer $k$.

(1) Suppose that $x\ge 0$ and let $x_N =\frac N{N-1} x$. If $x_N\ge 4$, then 
\begin{align*}\lim_{m \rightarrow \infty} \frac{1}{m} \log \E\, \mathcal{N}_{m, 2m -k,N}[x,\infty)&= \log (N-1) - \frac{x_N}{2}\left( 1- \frac 2N\right) - (k+1)\int_4^{x_N}\sqrt{\frac {t-4}{4t}}\, dt. \\
\lim_{m \rightarrow \infty} \frac{1}{m} \log \E\, \mathcal{N}_{m, 2m -k,N}[0,x) &= \log (N-1) - 2\left(1 - \frac 2N \right).\end{align*}
If $x_N\le 4$, then 
\begin{align*}
\lim_{m \rightarrow \infty} \frac{1}{m} \log \E\, \mathcal{N}_{m, 2m -k,N}[0, x) &= -\infty. \\
\lim_{m \rightarrow \infty} \frac{1}{m} \log \E\, \mathcal{N}_{m, 2m -k,N}[x,\infty) &=  \log (N-1) - 2\left(1 - \frac 2N \right).
\end{align*}

(2) If $k= k(m)$ such that $\frac{k}{m} \rightarrow\gamma\in (0,1)$, then
\begin{equation*}
		\lim_{m \rightarrow \infty} \frac{1}{m} \log \E\,\mathcal{N}_{m, 2m -k,N}(\R_{+}) = \log (N-1) - \left( 1 - \frac{2}{N} \right) \frac{s_{ \gamma}}{2},
\end{equation*}
where $s_\gamma$ is the number uniquely defined by the relation \rf{cond}.
\end{theorem}
	
The above results do not include the case $k(m) = m$.  However, in this case we can compute explicitly the expected value $\E\, \mathcal{N}_{m,m,N}(\R_+)$ and recover the formula in Baugher~\cite{BB}.  Define the density function $p_{m,m,N}$ by
$$\E \left[\sum_{z : \nabla s(z) = 0, \, \text{Ind}(\nabla^2s)(z) = k} f \left(\frac{1}{m+1}||s(z)||_h^2 \right)\right] = \int_{\R_+}f(x) p_{m,m,N}(x) dx$$
for any positive continuous function $f$ on $\R_+$.  Note that the above sum is simply the total number of critical points of Morse index $m$. Our second main result is an explicit formula for $p_{m,m,N}$. 
	
\begin{theorem}\label{distofsadd} For any $x \geq 0$,
		\begin{equation*}
		p_{m,m,N}(x) = (N-1)^m(m+1)^2 e^{-\frac{(m+1)N}{2(N-1)}(2 - \frac{2}{N}+m)x}.
		\end{equation*}
\end{theorem}

We can draw two consequences from this explicit density. 

\begin{corollary} For any $x \geq 0$,
		\begin{equation*}
		 \E\,\mathcal{N}_{m,m,N}[x,\infty) = \frac{2(N-1)^{m+1}(m+1)}{2 - \frac{2}{N}+m}e^{-\frac{(m+1)N}{2(N-1)}(2 - \frac{2}{N}+m)x}.
		\end{equation*}
\end{corollary}
\begin{proof}  Integrating the density function over $[x, \infty)$.
\end{proof}

For $x = 0$,  the above corollary recovers the formula
\begin{equation}\label{total}
\E\,\mathcal{N}_{m,m,N}(\R_+)= \frac{2(m+1)}{2(N-1)+mN}(N-1)^{m+1}
\end{equation}
proved in Baugher~\cite{BB}. For $x > 0$, it follows from the corollary that there exist positive constants $c_1$ and $c_2$ such that 
$\E\, \mathcal{N}_{m,m,N}(x,\infty) \leq c_1 e^{-c_2 m^2 x}$,
which shows that it becomes exponentially unlikely to find critical values away from 0 whose Morse index is $m$. 

The second consequence is that we can recover the exponential rate of $\E\, \mathcal{N}_{m,m+k,N}(\R_+)$ for any fixed $k>0$.

\begin{corollary}\label{Baugher}
	For a fixed $k > 0$, we have
	\begin{equation*}
	\lim_{m \rightarrow \infty}  \frac{1}{m} \log \E\, \mathcal{N}_{m,m+k,N} (\R_+) = \log(N-1).
	\end{equation*}
\end{corollary}
\begin{proof} According to  {\sc Theorem} 1.4 of Baugher~\cite{BB}, the total number of critical points $\mathcal{N}_{m,m+k,N}(\R_+)$ decreases as $k$ increases. Thus, given $\gamma \in (0,1)$ and $q(m)/m\rightarrow \gamma$, we have for large $m$,
\begin{equation*}
	\mathcal{N}_{m,2m-q(m),N} (\R_+)  \leq \mathcal{N}_{m,m+k,N}(\R_+) \leq \mathcal{N}_{m,m,N}(\R_+).
\end{equation*}	
For the right hand side, we have by (\ref{total})
$$\lim_{m\rightarrow\infty}\frac1m \log\E\, \mathcal{N}_{m,m,N}(\R_+) = \log (N-1).$$
For the left hand side, we have by the second part of {\sc Theorem} \ref{mainthm},
$$\lim_{m\rightarrow\infty}\log \E\, \mathcal{N}_{m,2m-q(m),N} (\R_+) = \log(N-1) - \left(1-\frac{2}{N} \right) \frac{s_{\gamma}}{2}.$$
We have $s_\gamma\rightarrow0$ as $\gamma\rightarrow1$, and the above limit reduces to that of the right hand side. The result follows immediately. 
\end{proof}

Finally, our third and last main result concerns the total number of critical points. 
	
\begin{theorem}\label{totalcrit} As before, we let $x \geq 0$ and $x_N = \frac{N}{N-1}x$. Then:
		\begin{equation*}
		\lim_{m \rightarrow \infty} \frac{1}{m} \log \E\, \mathcal{N}_{m,N} (x,\infty) = \begin{cases} \displaystyle{\log(N-1) - \left( 1 - \frac{2}{N} \right) \frac{x_N}{2} + \int_4^{x_N} \sqrt{\frac{4-t}{t}} dt}, & x_N \geq 4 \\
		\displaystyle{\log(N-1) - \left( 1 - \frac{2}{N} \right) \frac{x_N}{2}}, & x_N < 4 \end{cases}
		\end{equation*}
\end{theorem}

The remaining part of the paper is organized as follows. In {\sc Section} 2 we state some basic facts from complex geometry essential for the understanding of the paper. In {\sc Section} 3 we discuss the Wishart ensemble and its large deviations needed in the proof of the main results. {\sc Section 4} is devoted to explaining the relation between the expected number of critical points and the Wishart ensemble. The main results {\sc Theorem} \ref{mainthm} and {\sc Theorems} \ref{distofsadd} and \ref{totalcrit} are proved in {\sc Sections} 5. In {\sc Section} 6 we discuss the analogous case of random spherical harmonics.
\bigskip

\begin{Acknowledgements}.
			This material is based upon work supported by the National Science Foundation Graduate Research Fellowship. I would like to thank Antonio Auffinger and Steve Zelditch for inspiring conversations and unwavering patience. I would also like to thank my advisor Elton Hsu for his revision of this manuscript and for serving as a constant source of encouragement.
		\end{Acknowledgements}

\section{Complex projective space and line bundles}
In this section we recall some basic facts from complex geometry which are useful for understanding the setting of the paper.

The complex projective space $\C \Pa^m$ is the quotient space of $\C^{m+1}\backslash \{0\}$ by the equivalence relation 
$$\lambda (Z_0,...,Z_m) \sim (Z_0,...,Z_m) ,\quad \lambda \in \C^* = \C\backslash \{0\}.$$
This is a compact complex manifold with local charts $U_i = \{[Z_0, Z_1, ..., Z_m] | Z_i \neq 0 \}$ and trivializing maps $\Phi_i:\, U_i\rightarrow\C^m$ defined by
$$\Phi_i (Z) = (Z_0/Z_i,\ldots, \widehat{Z_i/Z_i}, \ldots, Z_m/Z_i).$$ 
We denote by $\mathcal{O}(N)$ the line bundle with the transition functions 
$$\sigma_{ij} : U_i \cap U_j \rightarrow \C^*, \quad\sigma_{ij}(Z) = \left( \frac{Z_i}{Z_j} \right)^N .$$
The sections of this bundle correspond to homogeneous holomorphic polynomials of degree $N$ in the variables $Z_0, ..., Z_m$. To see this, given a homogeneous holomorphic polynomial $p(Z_0,...,Z_m)$
we define the functions $f_{j}$ on $U_j$ by $f_{j}(Z) = p(Z/Z_j)$. It is easy to verify that  these functions glue up and yield a section on $\C \Pa^m$. Indeed, on the intersection $U_i\cap U_j$, we have
$$f_i(Z)\sigma_{ij}(Z) = p\left(\frac Z{Z_i}\right)\left( \frac{Z_i}{Z_j} \right)^N= p \left( \frac Z{Z_j}\right) = f_j(Z).$$ 
Conversely, a section is just a collection of polynomials $f_j$ on the charts $U_j$ satisfying $f_i(Z)\sigma_{ij}(Z) = f_j(Z)$ on the intersection $U_i\cap U_j$, which define a homogenous polynomial in a unique way by setting $p(Z) = Z_j^Nf_j(Z)$. 

We equip $\C \Pa^m$ with the Fubini-Study metric $h$ and denote the corresponding Chern connection on $\mathcal{O}(1)$ by $\nabla$. This induces canonically a connection on $\mathcal{O}(N)$, also denoted by $\nabla$, by requiring that it satisfy Leibniz's rule on tensors of sections. More explicitly, a section $s \in H^0(\C \Pa^m, \mathcal{O}(N))$ can be written locally as $s = f e^N$, where $e^N = \otimes_{i=1}^N e$ for a trivializing local frame $e$ for $\mathcal{O}(1)$ and a holomorphic function $f$ on a chart of $\C \Pa^m$. Then the connection $\nabla$ can be expressed explicitly as 
\begin{equation}
\nabla s = \sum_{j=1}^m(\partial_{z_j}f + f\,\partial_{z_j}K_N )\, dz_j \otimes e^N,\label{covariant}
\end{equation}
where $K_N$ is given by 
\begin{equation}
K_N = K_N(z , \bar{z}) = N \log(1 + |z|^2 ).\label{formulak}
\end{equation}
Since $\nabla$ also acts on 1-forms canonically, the Hessian $\nabla^2$ on holomorphic sections is well defined.  This action can be explicitly written in local coordinates as follows. For simplicity we introduce the notation $\nabla_{z_j} f := \partial_{z_j} f + f\,\partial_{z_j}K_N$ and $\nabla^2_{z_i, z_j}f = \nabla_{z_i} ( \nabla_{z_j} f)$. In the local basis $dz_i\otimes dz_j$, we can view $\nabla^2s$ as the $2m\times2m$ square matrix
\begin{equation*}
	 \nabla^2 s(z) =
	 \begin{bmatrix}
	 \nabla^2_{z_i, z_j} f &  f \Theta_N\\ \\ 
	 \overline{f \Theta_N} & \overline{\nabla^2_{z_i,z_j} f}
	 \end{bmatrix}
\end{equation*}
where $\Theta_N = \left\{ \partial^2_{z^i , \overline{z^j}}K_N\right\}$ . Note that this matrix is not Hermitian. For this reason, when discussing critical points of a section $s$, it is more convenient to use the real Hessian of 
$\log ||s(z)||^2_h$ by viewing $\C \Pa^m$ as a smooth manifold of real dimension 2$m$. By a slight abuse of notation, we use $\Ind(\nabla^2 s)(z)$ to denote the index of this matrix. From {\sc Lemma} 7.1 of Douglas, Shiffman, and Zelditch~\cite{DSZ}, we know that in local coordinates
$$ \Ind(\nabla^2 \log ||s(z)||^2_h ) = m + \Ind(\nabla^2_{z_i,z_j}f \,\Theta^*_N \overline{\nabla^2_{z_i,z_j}f} - \Theta_N ),$$
where $\Theta^*_N$ is the conjugate transpose of $\Theta_N$.
	
\section{The Wishart ensemble and related large deviations}
	
Let $X$ be a real $(m+1) \times m$ random matrix whose entries are i.i.d. Gaussians with mean zero variance $1/m$ and $W = X^TX$. We denote the law of $W$, the Wishart ensemble, by $\Pa_m$ and the corresponding expectation by $\E_m$. 
		
The only information we will need about the Wishart ensemble is the explicit distribution of its eigenvalues. For a vector $\lambda = (\lambda_1,...,\lambda_m)$, we define $\Delta(\lambda) = \prod_{i < j} (\lambda_i - \lambda_j)$, the Vandermonde determinant. We write the eigenvalues $\lambda = (\lambda_1, \ldots, \lambda_m)$ of $W$ in descending order, so that the vector $\lambda$ belongs to the region
$$\R^m_{\geq 0} = \{\lambda \in \R^m : \lambda_1 \geq ... \geq \lambda_m \geq 0 \}.$$
		
\begin{theorem}\label{wishart}
The joint density function of the decreasingly ordered eigenvalues of the Wishart ensemble with respect to the Lebesgue measure on $\R_{\ge 0}^m$ is 
$$\frac{1}{Z_{W}(m)}\Delta(\lambda) \exp \left(-\frac{m}{2}\sum_{i=1}^m \lambda_i \right),$$
where $Z_W(m)$ is the normalizing constant given by 
\begin{equation}\label{selb}Z_{W}(m) = 2^{m} m^{-m(m+1)/2} \prod_{j=1}^{m} j!\end{equation}
\end{theorem}
		
\begin{proof} See {\sc Theorem} 13.3.2 in Anderson~\cite{A} for the density, and {\sc Corollary} 2.5.9 of Anderson, Guionnet and Zeitouni~\cite{GZ} for the explicit formula for $Z_W(m)$. 
\end{proof} 

We now turn to the large deviations of the largest eigenvalues of the Wishart ensemble.  We will need the following large deviation principle for the $k$th largest eigenvalue under $\Pa_m$. 

\begin{theorem}\label{largedev}
		Under $\Pa_m$, the $k$th largest eigenvalue $\lambda_k$ satisfies the large deviation principle (LDP) with the speed $m$ and the good rate function $kI_{MP}$, where 
		
		$$I_{MP} (x)=\int_{4}^{x} \sqrt{\frac{t-4}{4t}} dt$$
for $x\ge 4$ and $\infty$ otherwise.
\end{theorem}
\begin{proof}
It is obvious that $I_{MP}$ is a good rate function. With this in mind, this theorem is equivalent to the following two assertions:
\begin{enumerate}
	\item $\limsup_{m \rightarrow \infty} \frac{1}{m} \log \Pa_m( \lambda_{k} \leq x) = -\infty$ for $0< x < 4$.
	\item $\lim_{m \rightarrow \infty} \frac{1}{m} \log \Pa_m( \lambda_{k} \geq x) = -k I_{MP}(x)$ for $x \geq 4$.
\end{enumerate}
For the proof, we need two previous results.

(a) Under the Wishart ensemble, the empirical measure 
$L_m  = \frac{1}{m}\sum_i \delta_{\lambda_i}$
of the eigenvalues satisfies an LDP with speed $m^2$. Its rate function is minimized uniquely at the Marchenko-Pastur distribution $\mu_{MP}$ on [0,4]
$$\mu_{MP}(dx) = \frac{1}{2 \pi} \frac{\sqrt{(4-x)x}}{x} \,dx.$$ 
This LDP is the content of {\sc Theorem} 5.5.7 of Hiai and Petz~\cite{HP}. 

(b) The functional
$$\phi(\mu,z) = \int_{\R_+} \log |z - y | \mu(dy) - \frac{z}{2}$$
defined on $\mathscr P(\R_+)\times\R_+$ is upper semi-continuous when we restrict it to $\mathscr P[0,M] \times [0,M]$ for any $M > 0$, and in fact it is continuous on $\mathscr{P}[0,r] \times [x,y]$ for $ y > x  > r \geq 4$,  see e.g. Auffinger, Ben Arous and \v{C}ern\'y~\cite{ABC}. Here $\mathscr P(A)$ is the space of probability measures on a set $A\subset\R_+$ with a metric compatible with the usual weak convergence of probability measures.  The distribution $\mu_{MP}$ and the rate function $I_{MP}$ are related through the functional by
\begin{equation}
\phi(\mu_{MP},x) = -I_{MP}(x) - 1.\label{mprelation}
\end{equation}
See Feral~\cite{DF}, page 48.

To prove assertion (1), we note that by definition, the inequality $\lambda_{k} \leq x$ for some $x < 4$ implies that $L_m[x,4] \leq (k-1)/m$.  Since $\mu_{MP}[x,4] > 0$,  there exists a closed set $C \subset \mathscr{P}(\R_+)$ such that $\mu_{MP} \notin C$ and 
$\{ \lambda_{k} \leq x \} \subset \{ L_m \in C \}$ for sufficiently large $m$. The LDP for $L_m$ recalled above implies that there exists a $c > 0$ such that 
$$\Pa_m( \lambda_{k} \leq x) \leq \Pa_m( L_m \in C) \leq Ke^{-cm^2},$$
which proves assertion (1).
	
To prove assertion (2), we first note that for the largest eigenvalue $\lambda_1$, 
\begin{equation}
\lim_{M \rightarrow \infty} \limsup_{m \rightarrow \infty} \frac{1}{m} \log \Pa_m( \lambda_1 > M) = - \infty,\label{neglect}
\end{equation}
which is precisely {\sc Lemma} 2.6.7 of Anderson, Guionnet and Zeitouni~\cite{GZ}.  Now we have 
$$\Pa_m( \lambda_{k} \geq x) \leq \Pa_m( \lambda_1 > M) + \Pa_m( \lambda_{k} \geq x, \lambda_1 < M).$$
In view of (\ref{neglect}), it is sufficient to show that for sufficiently large $M$,
\begin{equation}
\lim_{m \rightarrow \infty} \frac{1}{m} \log \Pa_m( \lambda_{k} \geq x, \lambda_1 < M)= -k I_{MP}(x).\label{ldp}
\end{equation}

We first prove the upper bound. We introduce new variables $\eta_i = \frac{m}{m-k} \lambda_i$ for $1 \leq i \leq m$ and write the density of $\Pa_m$ in terms of the $\eta_i$.  On the set $$ \{ x \leq \eta_k \leq \cdots  \leq \eta_1 < 2M\} \supset \left\{ x \leq \lambda_k \leq \cdots \leq \lambda_1< M\right\}$$ we have $\vert\eta_i-\eta_j\vert\le 2M$, and hence
\begin{align*}
\Pa_m(d\lambda) &= \frac1{Z_W(m)}\Delta (\lambda) \exp\left[-\frac m2\sum_{i=1}^m\lambda_i\right] d \lambda_1 \cdots d \lambda_m\\
&=\left(\frac{m-k}{m}\right)^{m(m+1)/2} \frac{1}{Z_W(m)}  \Delta(\eta) \exp \left[ -\frac{m-k}{2} \sum_{i=1}^m \eta_i \right] d \eta_1 \cdots d \eta_m\\
&\le \frac{(2M)^{(k-1)k/2}}{Z_W(m)}\left( \frac{m-k}{m} \right)^{m(m+1)/2} \prod_{i=1}^k \prod_{j=k+1}^m(\eta_i-\eta_j)\cdot\exp\left[-\frac{m-k}2\sum_{i=1}^k\eta_i\right]\,d\eta_1\cdots d\eta_k\times\\
&\qquad \qquad \qquad \qquad \qquad \qquad\qquad\prod_{k+1\le i<j\le m}(\eta_i-\eta_j)\cdot\exp\left[-\frac{m-k}2\sum_{i=k+1}^m\eta_i\right]d\eta_{k+1}\cdots d\eta_m\\
&=  (2M)^{(k-1)k/2}\left( \frac{m-k}{m} \right)^{m(m+1)/2}\frac{Z_W(m-k)}{Z_W(m)}\cdot d\eta_1\cdots d\eta_k \times \\
&\qquad \quad \qquad\exp\left[(m-k)\sum_{i=1}^k\phi(\tilde L_{m-k}, \eta_i)\right]\Pa_{m-k}(d\eta_{k+1}\cdots d\eta_m),
\end{align*}
where $\tilde L_{m-k}$ is the empirical distribution
$$\tilde L_{m-k}= \frac{1}{m-k} \sum_{i=1}^{m-k} \delta_{\eta_{k+i}}.$$
For $\epsilon > 0$, let $B_{\epsilon} \subset \mathscr P[0,M]$ be the ball of radius $\epsilon$ centered around $\mu_{MP}$ and $B^c_{\epsilon}$ its complement. On the set $\{x \leq \eta_k \leq \cdots \leq \eta_1 < 2M\}$, we have
$\exp \left[(m-k) \sum_{i=1}^k \phi(\tilde{L}_{m-k},\eta_i)\right] \leq (2M)^{k(m-k)}$ 
and thus \begin{equation*}
\exp \left[(m-k)\sum_{i=1}^k \phi(\tilde{L}_{m-k},\eta_i) \right] \leq  \exp \left[k(m-k)  \sup_{\mu \in B_{\epsilon}, y \in [x,2M]} \phi(\mu,y)\right] \mathbbm{1}_{B_{\epsilon}}(\tilde{L}_{m-k})+(2M)^{k(m-k)} \mathbbm{1}_{B^c_{\epsilon}}(\tilde{L}_{m-k}).
\end{equation*}
Integration over $\{x \leq \eta_k \leq \cdots \leq \eta_1 < 2M\}$ yields an upper bound for $\Pa_m(\lambda_{k} \geq x, \lambda_1 < M)$:
\begin{equation*}
\left( \frac{m-k}{m} \right)^{\frac{m(m+1)}{2}} (2M)^{\frac{k(k-1)}{2}} \frac{Z_W(m-k)}{Z_W(m)} \left( \exp \left[k(m-k)  \sup_{\mu \in B_{\epsilon}, y \in [x,2M]} \phi(\mu,y)\right] + (2M)^{k(m-k)} \Pa_{m-k}(\tilde{L}_{m-k} \notin B_{\epsilon})\right).
\end{equation*}
Two observations are in order. The first observation is that $\tilde{L}_{m-k}$ with respect to $\Pa_{m-k}$ satisfies the same LDP as $L_m$ with respect to $\Pa_m$. In particular, this implies that for $m$ large enough there exists a $c > 0$ for which \begin{equation*}
\Pa_{m-k}(\tilde{L}_{m-k} \notin B_{\epsilon}) \leq \exp (-c m^2),
\end{equation*}
hence the probability $\Pa_{m-k}(\tilde{L}_{m-k} \notin B_{\epsilon}) $ is negligible in the limit. The second observation is that by use of (\ref{selb}), one can compute 
\begin{equation*}
\lim_{m \rightarrow \infty} \frac{1}{m} \log \left[\left( \frac{m-k}{m} \right)^{m(m+1)/2}\frac{Z_W(m-k)}{Z_W(m)}\right] = k.
\end{equation*} 
In light of these two observation, we arrive at the inequality
$$\limsup_{m \rightarrow \infty} \frac{1}{m} \log \Pa_m( \lambda_{k} \geq x) \leq k  + k\lim_{\epsilon \downarrow 0} \sup_{\mu \in B_{\epsilon},y \in [x,2M]} \phi(\mu,y).$$
The second term can be computed explicitly,  
$$\lim_{\epsilon \downarrow 0} \sup_{\mu \in B_{\epsilon},y \in [x,2M]} \phi(\mu,z) = \sup_{y \in [x,2M)} \phi( \mu_{MP},y) 
= -I_{MP}(x) - k,$$
where the first equality follows from the upper-semicontinuity of $\phi$ and the second equality follows from (\ref{mprelation}) and the monotonicity of $I_{MP}$.
	
To obtain the lower bound, fix $y > x > r \geq 4$ and $\epsilon > 0$. We retain the definition of the $\eta_i$ as in the proof of the upper bound, and and on the set 
$$ \left \{ y \geq \eta_1 \geq \cdots \geq \eta_k \geq \frac{m}{m-k}x \right \} = \left \{ \frac{m-k}{m}y \geq \lambda_1 \geq \cdots \geq \lambda_k \geq x  \right\}  \subset \{ \lambda_k  \geq x \}$$ 
we can produce the inequality
\begin{align*}
\Pa_m(d\lambda) &= \frac1{Z_W(m)}\Delta (\lambda) \exp\left[-\frac m2\sum_{i=1}^m\lambda_i\right] d \lambda_1 \cdots d \lambda_m\\
&\geq  \left( \frac{m-k}{m} \right)^{m(m+1)/2}\frac{Z_W(m-k)}{Z_W(m)} \prod_{1 \leq i < j \leq k} \vert \eta_i - \eta_j|\cdot d\eta_1\cdots d\eta_k \times \\
&\qquad \mathbbm{1}_{B_{\epsilon} \cap \mathscr{P}[0,r]}(\tilde{L}_{m-k})\exp\left[k(m-k) \inf_{\mu \in B_{\epsilon} \cap \mathscr{P}[0,r], z \in [x,y] }\phi(\mu, z)\right]\Pa_{m-k}(d\eta_{k+1}\cdots d\eta_m).
\end{align*}
where by $B_{\epsilon} \cap \mathscr{P}[0,r]$, I mean the set of measures in $B_{\epsilon}$ whose support is contained in $[0,r]$. By integrating over $\left \{ y \geq \eta_1 \geq  \cdots  \geq \eta_k \geq \frac{m}{m-k}x \right \}$, we obtain
	\begin{align*}
	\Pa_m( \lambda_{k} \geq x) 	&\geq \Pa_m \left( y \geq \eta_1 \geq  \cdots  \geq \eta_k \geq \frac{m}{m-k}x  \right) \\
	&\geq \left( \frac{m-k}{m} \right)^{m(m+1)/2}\frac{Z_W(m-k)}{Z_W(m)} \int \prod_{1 \leq i < j \leq k} | \eta_i - \eta_j| d \eta_1 \cdots d \eta_k \times \\
	& \qquad \qquad\exp\left[k(m-k) \inf_{\mu \in B_{\epsilon} \cap \mathscr{P}[0,r], z \in [x,y] }\phi(\mu, z)\right] \Pa_{m-k}(\tilde{L}_{m-k} \in B_{\epsilon} \cap \mathscr{P}[0,r]) 
	\end{align*}where the inner integral is over the set$$\left \{ y \geq \eta_1 \geq \cdots \geq \eta_k \geq \frac{m}{m-k}x \right \}.$$ 
The inner integral is bounded away from zero and from above, so it will have no effect in the limit. The factor $\Pa_{m-k} (\tilde{L}_{m-k} \in B_{\epsilon} \cap \mathscr{P}[0,r])$ converges to one by the previously mentioned LDP, hence it too will not affect the limit. It follows that in the limit the inequality becomes
$$\liminf_{m \rightarrow \infty}\frac{1}{m} \log \Pa_m( \lambda_k \geq x) \geq k + k \lim_{\epsilon \downarrow 0}\inf_{\mu \in B_{\epsilon} \cap \mathscr{P}[0,r], z \in [x,y]} \phi(\mu,z).$$
We use the continuity of $\phi$ and (\ref{mprelation}) to obtain 
$$\lim_{\epsilon \downarrow 0} \inf_{\mu \in B_{\epsilon} \cap \mathscr{P}[0,r],z \in [x,y]} \phi(\mu,z) = -I_{MP}(y)-1.$$
Finally, we let $y \rightarrow x$ and use the continuity of $I_{MP}$ obtain our desired result.
\end{proof}
		
\section{Expected number of critical  points and the Wishart ensemble}
 In this section we relate $\E\,\mathcal{N}_{m,2m-k,N}(B)$ to  the $(k+1)$th largest eigenvalue of an $(m+1) \times (m+1)$ Wishart matrix.
 
 \begin{theorem}\label{aux} For a Borel set $B \subset \R_+$,
 	\begin{equation}\label{complexindex}
 	\E\, \mathcal{N}_{m,2m-k,N}(B) = \frac{2(N-1)^{m+1}}{N}\, \E_{m+1}\left [e^{-(1-\frac{2}{N})\frac{m+1}{2} \lambda_{k+1}}; \ \lambda_{k+1}\in \frac N{N-1}B\right].
 	\end{equation}
 \end{theorem}

The proof of this identity is based on the following Kac-Rice formula adapted to our setting. 
	
\begin{proposition}\label{kacrice}
Let $\rho_{\nabla s(z)}$ denote the probability density function of $\nabla s(z)$ as a (random) vector in $\C^m$ (see (\ref{covariant})). Then $\E\,\mathcal{N}_{m,2m-k,N}(B)$ equals
$$ \int_{\C \Pa^m} \rho_{\nabla s(z)}(0) \E [ | \det \nabla^2 s(z) | \mathbbm{1}_{(m+1)B}(||s(z)||_h^2) \mathbbm{1}_{\textnormal{Ind}\nabla^2 s(z)= 2m-k} | \nabla s(z) = 0 ] v(dz).$$
\end{proposition}
\begin{proof} See {\sc Theorem} 4.4 of Douglas, Shiffman, and Zelditch~\cite{DSZ1}.
\end{proof} 

\begin{remark}	In general $\rho_{\nabla s(z)}$ depends on our choice of $s^N_i$. Nevertheless, its value $\rho_{\nabla s(z)}(0)$ at the origin is independent of the choice.
\end{remark}

By $SU(m+1)$-invariance, the integrand in the above Kac-Rice formula is independent of $z$, thus the $z$-integration can be replaced by the multiplication of $\text{vol}(\C \Pa^m)$ and we need to evaluate the expectation at the point $z =0$.  For this purpose, we write $s(z) = f(z) e^N$ in local coordinates  near the point $z =0$. We have $\nabla_{z_i}f = \partial_{z_i}f := \partial_i f$ at $z = 0$. 
	
\begin{lemma}\label{covar} The covariance of $f$ and its first and second derivatives at $z = 0$ are given as follows.
\begin{align*}
		\E[f(0) \overline{f(0)}] &= 1,\\
		\E[ f(0) \overline{\partial_{i}f(0)}] &=0,\\
                  \E[f(0) \overline{\partial_{i} \partial_{j}f(0)}] &= 0, \\ 
                  \E[\partial_if(0) \overline{\partial_j f(0)}] &= N \delta_{ij}, \\
                  \E[\partial_if(0)\overline{ \partial_j\partial_k f(0)}] &= 0, \\
		\E[\partial_{i} \partial_{j} f(0) \overline{\partial_k\partial_l f(0)}] &= N(N-1) (\delta_{il}\delta_{jk} + \delta_{ik} \delta_{jl}).
\end{align*}
\end{lemma}
\begin{proof} The Gaussian field defined in \rf{eq:Gauss} is uniquely determined by its covariance kernel
$$\E[ s(x) \otimes \overline{s(y)}] = \frac{N! \pi^m}{(N+m)!} \Pi_{N,m}(x,y).$$ 
Here $\Pi_{N,m}$ is the kernel of the projection from $L^2(\C \Pa^m, \mathcal{O}(N))$ into $H^0(\C \Pa^m, \mathcal{O}(N))$. Note that this kernel is independent of our choice of an orthonormal basis $s^N_i$ in (\ref{eq:Gauss}). In local coordinates, it can be explicitly written as
$$\frac{N! \pi^m}{(N+m)!} \Pi_{N,m}(x,y) = (1 + z \cdot \bar{w})^N e^N(z) \otimes \overline{e^N(w)},$$
where $z$ and $w$ are the (inhomogeneous) coordinates of $x$ and $y$. The covariances in the statement follows by straightforward computations.
\end{proof}

As immediate consequences of {\sc Lemma} \ref{covar}, we see that $\rho_{\nabla f(0)}(0) = 1/(N\pi)^m$ and that both the matrix $\partial^2_{ij}f(0)$ and $f(0)$ are independent of the event $\partial_kf(0$, hence also independent of $\nabla s(0) = 0$. From (\ref{formulak}) we have $\partial^2_{z_i,\overline{z_j}}K_N(0) = N \delta_{ij}$, hence from (\ref{covariant}) we have
$$\det \nabla^2 s(0)  = \det(YY^* - N^2|f(0)|^2 I_m),$$ 
where the matrix $Y =\left\{ \partial^2_{ij}f(0)\right\}$ and  $I_m$ is the $m \times m$ identity matrix. Obviously the value of the determinant depends only on the eigenvalues of $YY^*$.  Therefore we need to study the distribution of the eigenvalues of $YY^*$, which is a $\Sym(m, \C)$-valued random matrix.

\begin{proposition} The law of the eigenvalues of $W = YY^*/mN(N-1)$ is identical with the law of the eigenvalues under the Wishart ensemble.
\end{proposition}
\begin{proof} The natural Lebesgue measure on $\Sym(m,\C)$ as a real vector space is 
$$dH = \prod_{i \leq j} \textnormal{Re } dH_{ij} \textnormal{ Im } dH_{ij}.$$ 
From the last covariance identification in {\sc Lemma} \ref{covar} the density function of $Y$ with respect to the Lebesgue measure $dH$ is 
\begin{equation}
\frac{1}{2^m(N(N-1)\pi)^{\frac{m(m+1)}{2}}} e^{-\frac{1}{2N(N-1)}\text{Tr}(HH^*)}.\label{ydensity}
\end{equation}
Define the map $\varPhi : U(m) \times \R^m_{\geq 0} \rightarrow \text{Sym}(m,\C)$ by
$$\varPhi(U, \lambda) = U\text{diag}(\sqrt{\lambda}) U^T,$$ 
where $\text{diag}(\sqrt{\lambda})$ the diagonal matrix whose entries are $\sqrt{\lambda_1}, ..., \sqrt{\lambda_m}$ and $U^T$ is the transpose of $U$. ByTakagi's factorization (see {\sc Corollary} 4.4.4 of Horn and Johnson~\cite{H}), almost every $X \in \Sym (m, \C)$  can be written uniquely as $X = U \text{diag}(\sqrt{\lambda(XX^*)} U^T$,
where $U$ is a unitary matrix and $\lambda_i(XX^*)$ are the eigenvalues of $XX^*$ in decreasing order. A well known computation shows that the image of the Lebesgue measure $dH$ under $\varPhi$ becomes $\varPhi_*(dH) = \Delta(\lambda)\, d\lambda\,dU$, where $dU$ is the properly normalized Haar measure on $U(m)$.  Note that the Jacobian in this case is $\Delta(\lambda)$, a function of $\lambda$ alone. On the other hand, the exponent in (\ref{ydensity}) is 
$$\frac1{N(N-1)}\text{Tr}(YY*) = \frac1{N(N-1)}\sum_{i=1}^m\lambda_i(YY^*) = m\sum_{i=1}^n\lambda_i(W).$$
By passing from $\Sym(m, \C)$ to $U(m)\times\R^m_{\ge0}$, we see from (\ref{ydensity}) that the density functions for the distribution of the eigenvalues of $W = YY^*/N(N-1)$ must be a constant multiple of
$\Delta(\lambda)\exp\left[ -\frac m2\sum_{i=1}^m\lambda_i\right]$. Comparing this with the density function of the eigenvalues under the Wishart ensemble in {\sc Lemma} \ref{wishart} we obtain the result immediately. 
\end{proof}

Summarizing what we have proved so far, from the Kac-Rice formula in {\sc Proposition} \ref{kacrice} we conclude that $\E\,\mathcal{N}_{m,2m-k,N}(B)$ equals
\begin{equation*}\label{eq:mess}
	\frac{\textnormal{Vol}(\C \Pa^m) m^m (N-1)^m}{\pi^m}\, \E \left[ \mathbbm{1}_{(m+1)B}(|f(0)|^2) \mathbbm{1}_{[\lambda_{k+1}, \lambda_k]}\left(\frac{N|f(0)|^2}{(N-1)m}\right)\prod_{i=1}^m \left|\lambda_i- \frac{N|f(0)|^2}{(N-1)m}\right| \right],
	\end{equation*}
where $\lambda_i =\lambda_i(W)$ with $W$ obeying the Wishart ensemble and $f(0)$ is, according to {\sc Lemma} \ref{covar}, a standard complex Gaussian random variable independent of $W$. 
It remains to identify this with (\ref{complexindex}).  For this purpose, we note that $\frac{N |f(0)|^2}{(N-1)m}$ is exponentially distributed with mean $ \frac{N}{(N-1)m}$. Thus the expectation is 
\begin{equation}\label{lintegral}
\frac{m(N-1)}{NZ_W(m)} \int_{\frac{N(m+1)}{m(N-1)}B}  \int  \prod_{i=1}^m |\lambda_i - x| \Delta(\lambda) e^{-\frac{m}{2}(1-\frac{2}{N}) x }e^{-\frac{m}{2}(\sum_{i=1}^m \lambda_i + x)} d \lambda dx,
\end{equation}
where the inner integral with respect to $\lambda$ is over the set 
$$\{\lambda_1 > ... > \lambda_k > x > \lambda_{k+1} > ... > \lambda_m > 0\}.$$
This domain suggests we treat $x$ as if it is another $\lambda$. More precisely, introduce the new variables $\mu_i = \lambda_i$ for $1\le i\le k$, $\mu_{k+1} = x$, and $\mu_i = \lambda_{-1}$ for $k+2\le i\le m$. For the Vandermonde polynomial we have $\Delta(\mu) = \Delta(\lambda) \prod_{i=1}^m | \lambda_i - x|$.  In terms of the new variables $\mu$, the integral (\ref{lintegral}) becomes
$$\frac{m(N-1)}{NZ_W(m)}\int_{\R^{m+1}_{\geq 0}}  \mathbbm{1}_{\frac{(N-1)(m+1)}{mN}B} (\mu_{k+1}) e^{- (1 - \frac{2}{N}) \frac{m}{2} \mu_{k+1}} \exp \left[- \frac{m}{2}\sum^{m+1} _{i=1} \mu_i \right] \Delta(\mu) d \mu.$$
Comparing this with {\sc Lemma} \ref{wishart}, this is exactly the expectation with respect to $\Pa_{m+1}$ up to a constant. We will omit the identification of the constant stated in the theorem, it being a straightfoward computation using Selberg's integral formula for $Z_W(m)$. This completes the proof {\sc Theorem} \ref{aux}, our main result of this section. 

An immediate consequence of {\sc Theorem} \ref{aux} is that $\E\, \mathcal{N}_{m,q+1,N} (\R_{+})$ is decreasing in $q$ in the range $m \leq q < 2m$,
agreeing with {\sc Theorem} 1.4 of Baugher~\cite{BB}. Also, summing over $k$ in (\ref{complexindex}), we obtain the following corollary.

\begin{corollary}\label{totaux}
\begin{equation*}
\E\,\mathcal{N}_{m,N}(B) = \frac{2(m+1)(N-1)^{m+1}}{N} \int_{\frac{N}{N-1}B} e^{-(1-\frac{2}{N})\frac{m+1}{2}x} p_{m+1}(x) dx.
\end{equation*}
\end{corollary}
Here $p_{m+1}$ is the density function of the expected empirical distribution of the eigenvalues of the Wishart ensemble; namely, 
for any bounded continuous function $f$,
$$\E_{m+1}\left[ \frac{1}{m+1}\sum_{i=1}^{m+1}f(\lambda_i)\right] = \int_{\R_{+}} f(x) p_{m+1}(x) dx$$

\section{Proof of the main results}

In this section we prove our main results stated in {\sc Section} 1. 

\subsection{Proof of Theorem \ref{mainthm}}

 {\sc Theorems} \ref{aux} and \ref{largedev} together with Varadhan's lemma (see {\sc Theorem} 4.3.1 of Dembo and Zeitouni~\cite{DZ}) imply the first part of {\sc Theorem} \ref{mainthm}. The second part of {\sc Theorem} \ref{mainthm} is a straightforward corollary of the following lemma.

\begin{lemma} \label{concentration} For any $\epsilon > 0 $, $\gamma \in (0,1)$ and $\frac{k(m)}{m} \rightarrow \gamma$ , there exists a constant $C=C(\epsilon)$ such that
$$\Pa_m( \lambda_{k(m)} \notin (s_{\gamma} -\epsilon, s_{\gamma} + \epsilon) ) \leq e^{-Cm^2},$$
where $s_{\gamma}$ is defined as in (\ref{cond}).
\end{lemma}
\begin{proof}
This is an immediate consequence of the large deviation principle for $L_m = \frac1m\sum_{i=1}^m \delta_{\lambda_i}$ with respect to $\Pa_m$ whose rate function is minimized at the Marchenko-Pastur distribution $\mu_{MP}$ (see {\sc Theorem} 5.5.7 of Hiai and Petz~\cite{HP}). To see this, we use the fact that 
$$\Pa_m( \lambda_{k(m)} > s_{\gamma} + \epsilon) = \Pa_m \left(L_m(s_{\gamma} + \epsilon,\infty) \geq \frac{k(m)}{m} \right).$$
Since   $\mu_{MP}(s_{\gamma}+\epsilon, \infty) <  \mu_{MP}(s_{\gamma}, \infty) = \gamma$, there must exist a positive constant $C$ such that for large $m$ $$\Pa_m \left(L_m(s_{\gamma} + \epsilon,\infty) \geq \frac{k(m)}{m} \right) \leq  \exp(-Cm^2 )$$An analogous argument can be made for $\Pa_m(\lambda_{k(m)} < s_{\gamma} - \epsilon))$ which we leave to the reader.
\end{proof}

\subsection{Proof of Theorem \ref{distofsadd}} 
{\sc Theorem} \ref{distofsadd} is equivalent to the statement that for any Borel set $B$, $$\E \, \mathcal{N}_{m,m,N}(B) = \int_{B} (N-1)^m(m+1)^2 e^{-\frac{(m+1)N}{2(N-1)}(2 - \frac{2}{N}+m)x} dx$$	The crux of the proof lies in the following

\begin{lemma}\label{eq:smalleigdist}
The distribution of the smallest eigenvalue $\lambda_m$ of the Wishart ensemble given by
$$\Pa_m \left( \frac{m}{2} \lambda_m \geq x \right) = e^{-mx}.$$
\end{lemma}
\begin{proof} This is {\sc Theorem} 4.2 of Edelman~\cite{EA} but for we provide a short proof here. We have
$$\Pa_m \left( \frac{m}{2} \lambda_m  \geq x \right) = \frac{1}{m! Z_W(m)}\int_{\frac{2x}{m}}^{\infty} ... \int_{\frac{2x}{m}}^{\infty} \Delta(\lambda) \exp \left( -\frac{m}{2} \sum_{i=1}^m \lambda_i \right) d \lambda.$$
Making a change of variable $\mu = \lambda - \frac{2x}{m}$ we see that the probability must be of the form of a constant times $e^{-mx}$, hence the result.. 
\end{proof}

Returning to the proof of {\sc Theorem} \ref{distofsadd}, we recall from (\ref{complexindex}) that
$$\E \, \mathcal{N}_{m,m,N}(B) = \frac{2(N-1)^{m+1}}{N} \E_{m+1} \left[e^{-(1 -\frac{2}{N})\frac{m+1}{2} \lambda_{m+1}}  ; \lambda_{m+1} \in \frac{N}{N-1}B \right].$$
{\sc Lemma} \ref{eq:smalleigdist} allows us to write 
\begin{align*}\E_{m+1}\left[ e^{-(1- \frac{2}{N})\frac{m+1}{2}\lambda_{m+1}} ;\lambda_{m+1} \in \frac{N}{N-1}B \right] &= \int_{\frac{(m+1)N}{2(N-1)}B} (m+1) e^{-(1-\frac{2}{N})x} e^{-(m+1)x} dx\\
			&= \int_{B}\frac{(m+1)^2N}{2(N-1)} e^{-\frac{(m+1)N}{2(N-1)}(2 - \frac{2}{N}+m)u} du, 
\end{align*}
where the second equality follows from the change of variables $u = \frac{2(N-1)}{(m+1)N}x$. Since this is true for any Borel set $B$, we obtain the desired result.
	
\subsection{Proof of Theorem \ref{totalcrit}}
To simplify the notation, we introduce
$$\psi(t) = \log(N-1) - (1-\frac{2}{N})\frac{t}{2}.$$
We first consider the case $x_N \geq 4$.  We have the following inequalities: 
$$\frac{2}{N}\E_{m+1}[e^{(m+1)\psi(\lambda_{1})} ; \lambda_1 \geq x_N] \leq \frac{2(m+1)}{N}\int_{x_N}^{\infty} e^{(m+1)\psi(t)} p_{m+1}(t) dt \leq \frac{2(m+1)}{N} e^{(m+1) \psi(x_N)} \Pa_{m+1}(\lambda_1 \geq x_N).$$
By {\sc Corollary} $\ref{totaux}$, the middle expression is $ \E \, \mathcal{N}_{m,N}(x,\infty)$. For the right hand side, {\sc Theorem} \ref{largedev} yields 
\begin{equation*}
\lim_{m \rightarrow \infty}\frac{1}{m} \log \left[\frac{2(m+1)}{N} e^{(m+1)\psi(x_N)} \Pa_{m+1}(\lambda_1 \geq x_N)\right]= \psi(x_N) + I_{MP}(x_N).
\end{equation*}
For the left hand side, we apply Varadhan's lemma ({\sc Theorem} 4.3.1 of Dembo and Zeitouni~\cite{DZ}) in conjunction with {\sc Theorem \ref{largedev}} to obtain 
$$\lim_{m \rightarrow \infty}\frac{1}{m} \log \left[ \frac{2}{N}\E_{m+1}[e^{(m+1)\psi(\lambda_{1})} ; \lambda_1 \geq x_N] \right] = \psi(x_N) + I_{MP}(x_N) .$$
The use of Varadhan's lemma is justified because 
$\psi$ is bounded from above
and thus the tail condition in {\sc Theorem} 4.3.1 of Dembo and Zeitouni~\cite{DZ}) is satisfied.

We now consider the case $x_N < 4$. 
We can use the same inequality we used in the case $x_N \geq 4$ for the upper bound. 
Unfortunately, the lower bound given by this inequality is not sharp enough. To remedy this defect, we use a different inequality
$$\frac{2}{N}e^{(m+1)\psi(x_N+\epsilon)} \Pa_{m+1}( L_{m+1}[x_N, \infty) > 0) \leq \frac{2(m+1)}{N}\int_{x_N}^{\infty} e^{(m+1) \psi(t)} p_{m+1}(t) dt,$$
which holds for any positive $\epsilon$. 
The LDP on $L_m$ guarantees that $\Pa_m(L_m[x_N, \infty) > 0 ) \rightarrow 1$ since the rate function for this LDP is minimized at the Marchenko-Pastur distribution on [0,4], which assigns positive measure to $[x_N,\infty)$. Hence, $$\lim_{m \rightarrow \infty} \frac{1}{m} \log \left( \frac{2}{N}e^{(m+1) \psi(x_N + \epsilon)} \Pa_{m+1}(L_{m+1}[x_N,\infty)>0) \right) = \psi(x_N + \epsilon).$$ Since $\epsilon$ is arbitrary and $\psi$ is continuous, we are done.

\section{Spherical harmonics}
The case of spherical harmonics of degree $N$ on the sphere $S^m$ is similar to the case of holomorphic 
sections of the line bundle $\mathcal{O}(N)$ over $\C \Pa^m$. To be precise, we define the spherical harmonics of degree $N$ by considering the space $H_{N}(S^m)$ of homogenous harmonic polynomials on $\R^{m+1}$ and viewing the functions in $H_N(S^m)$ as functions on the sphere $S^m$ by restriction. We view $H_N(S^m)$ as a Hilbert space equipped with the $L^2(S^m)$ inner product and choose an orthonormal basis $\varphi_{i,N}$. The Gaussian field of random spherical harmonics of degree $N$ is 
 \begin{equation} \label{varphi}\varphi = \sum_{i} c_i \varphi_{i,N}
 \end{equation}
where the $c_i$ are i.i.d. mean zero Gaussians with the normalized variance
$$\E |c|^2 = \frac{\text{Vol}(S^m)}{ \dim H_N(S^m)}.$$
With this choice of normalization, we have $$\E[\varphi(x) \varphi(y)] = \frac{\text{Vol}(S^m)}{ \dim H_N(S^m)} P_{N,m}(x,y) := \nu_{N,m}( \langle x, y \rangle) $$
where $P_{N,m}$ is the projection kernel from $L^2(S^m) \rightarrow H_N(S^m)$, as in the complex case, and $\nu_{N,m}$ is a real-valued function. The key property of this covariance kernel is that it only depends on the inner product $\langle x,y \rangle$ and hence it is invariant under the usual $SO(m+1)$ action on $S^m$, similar to the $SU(m+1)$ invariance of the covariance kernel in (\ref{eq:Gauss}).
 
 For a set $B \subset \R$, we define $\mathcal{N}_{m,k,N}(B)(\varphi)$ to be the number of critical points of $\varphi$ of Morse index $k$ with values in $\sqrt{m+1}B$. Symbolically, 
$$\mathcal{N}_{m,k,N}(B) (\varphi)= \sum_{\sigma : \nabla \phi(\sigma) = 0, \textnormal{Ind}( \nabla^2 \varphi(\sigma)) = k } \mathbbm{1}_{\sqrt{m+1}B}(\varphi(\sigma)),$$
where $\nabla$ and $\nabla^2$ denote the standard gradient and Hessian in the ambient space $\R^{m+1}$ restricted to the sphere $S^m$ and $\textnormal{Ind}(\nabla^2 \varphi(\sigma))$ is the number of negative eigenvalues of $\nabla^2 \varphi (\sigma)$. As before, we can view this as an integer-valued random variable if we sample $\varphi$ according to the Gaussian field (\ref{varphi}). The number of critical points is 
$\mathcal{N}_{m,N}(B) = \sum_k \mathcal{N}_{m,k,N}(B)$.  We have the following results analogous to those stated in our {\sc Theorem} \ref{mainthm}.

\begin{theorem}For a fixed integer $k$, we have:
\begin{equation*}
\lim_{m \rightarrow \infty }\frac{1}{m} \log \E\,  \mathcal{N}_{m,k,N}(\R) = \frac{1}{2} \log(N-1) - (1 - \frac{2}{N})
\end{equation*}
If we make no restrictions on the Morse index, we have:
\begin{equation*}
\lim_{m \rightarrow \infty }\frac{1}{m} \log \E\, \mathcal{N}_{m,N}(\R) = \frac{1}{2} \log (N-1)
\end{equation*}
\end{theorem}

These results take the form as those in {\sc Theorem} \ref{mainthm} except for the factor 1/2. Most of the computations required for the proof of this theorem can be found in Auffinger and Ben Arous~\cite{AB} with necessary changes. One of the differences needing to be taken care of is that our covariance kernel is not given by a single positive-definite function independent of the dimension $m$.

\end{document}